\documentclass[11pt,british]{amsart}

\usepackage{ifpdf}
\ifpdf\relax\else
\usepackage[T1]{fontenc}
\fi

\usepackage{babel,eucal,url,amssymb,stmaryrd,
enumerate,amscd,
}

\usepackage[pagebackref]{hyperref}

\theoremstyle{plain}
\newtheorem{thm}{Theorem}[section]
\newtheorem{lem}[thm]{Lemma}
\newtheorem{pro}[thm]{Proposition}

\theoremstyle{definition}
\newtheorem{defn}[thm]{Definition}

\newtheorem{rem}[thm]{Remark}

\newcommand{\Gtwo}{\ifmmode{{\rm G}_2}\else{${\rm G}_2$}\fi}

\date{\today}

\begin{document}

\title[ParaSasakian manifolds ...]{ParaSasakian manifolds with constant paraholomorphic sectional curvature}

\author{Simeon Zamkovoy}
\address[Zamkovoy]{University of Sofia "St. Kl. Ohridski"\\
Faculty of Mathematics and Informatics\\
Blvd. James Bourchier 5\\
1164 Sofia, Bulgaria} \email{zamkovoy@fmi.uni-sofia.bg}

\begin{abstract}
{In this paper paraSasakian manifolds with constant
paraholomorphic sectional curvature are considered.}

\end{abstract}

\maketitle \setcounter{tocdepth}{3} \tableofcontents

\section{Introduction}

A Riemannian manifold is said to admit the axiom of planes if
there exists a 2-dimensional totally geodesic submanifold tangent
to any 2-dimensional section at every point of the manifold. It is
said to admit the free mobility if there exists an isometry which
carries any point and any frame attached to the point to any other
point and any other frame attached to the point \cite{Car}.

It is well known that a Riemannian manifold satisfies the axiom of
planes or the free mobility if and only if it is of constant
curvature \cite{Car}. A similar result for complex manifolds is
proved in \cite{YM}. More precisely, a K\"{a}hler manifold with
constant holomorphic curvature admits either of the so-called
axiom of holomorphic planes and the holomorphic free mobility, and
conversely. For a para-K\"{a}hler manifold, the concept of
paraholomorphic free mobility and the axiom of paraholomorphic
planes are given in \cite{GB,RAG}.

Some results on manifolds with constant nonvanishing
paraholomorphic sectional curvature are given in \cite{B} and
\cite{GB}. The general expression for the metric and the  almost
product structure of the para-K\"{a}hler space form in normal
coordinates is given in \cite{GA}.

The geometry of almost contact manifolds is a natural extension of
the almost Hermitian geometry in the odd dimensional case. Almost
contact manifolds satisfying the the axiom of planes or the axiom
of free mobility are considered in \cite{O}. It is  proved that a
Sasakian manifold satisfies the axiom of $C$-holomorphic planes or
the one of $C$-holomorphic free mobility if and only if it is of
constant $C$-holomorphic curvature. Similarly, the geometry of
almost paracontact manifolds is a natural extension of the almost
paraHermitian geometry in the odd dimensional case \cite{Z1,S1}.

A paracontact metric manifold $M$ is a (2n+1)-dimensional manifold
equipped with a 1-form $\eta$, a tensor field $\varphi$ of type
$(1,1)$, a vector field $\xi$ satisfying $\eta (\xi)=1$,
$\varphi\xi=0$ and $\varphi^2 = id - \eta \otimes \xi$, and a
compatible metric $g$ satisfying $g(\varphi X,\varphi
Y)=-g(X,Y)+\eta (X)\eta (Y)$ and $g(X,\varphi Y)=d\eta(X,Y)$,
where $d\eta(X,Y)=\frac12(X\eta(Y)-Y\eta(X)-\eta([X,Y]))$. The
manifold $M^{(2n+1)}(\varphi,\xi,\eta,g)$ is said to be a
paraSasakian manifold if in addition
$N_{\varphi}(X,Y)=2d\eta(X,Y)\xi$, where
$$N_{\varphi}(X,Y)=[\varphi X,\varphi Y]-\varphi [\varphi
X,Y]-\varphi [X,\varphi Y]+\varphi^2[X,Y].$$

An example of paracontact metric manifold take the hyperbolic
Heisenberg group {$\mathcal{H}^{2n+1} =\mathbb{R}^{2n}\times
\mathbb{R}$  \cite{S1}, where the group law is given by
\[
( p'',t'')\ =\ (p',t')\circ(p, t)\ =\ \bigl (p'\ +\ p, t' +\ t\ -
\sum_{k=1}^n(u'_kv_k-v'_ku_k)\bigr),
\]
\noindent for $p',p\in\mathbb R^{2n}$, $t', t\in \mathbb R$,
$p=(u_{1},v_{1},\dots,u_{n},v_{n})$ and
$p'=(u'_{1},v'_{1},\dots,u'_{n},v'_{n})$. Take $U_{k}=${\
}$\frac{\partial }{\partial u_{k}}+v_{k} \frac{\partial }{\partial
t},$ $V_{k}=\frac{\partial }{\partial v_{k}}-u_{k} \frac{\partial
}{\partial t},\xi =2\frac{\partial }{\partial t}$ as a basis of
left-invariant vector fields. Define $ \eta_{\mathcal{H}^{2n+1}}\
=\ \frac{1}{2}dt\
{+}\frac{1}{2}\sum_{k=1}^{n}(u_{k}dv_{k}-v_{k}du_{k})$, the tensor
field $\varphi$ of type $(1,1)$ by $\varphi U_{k}=V_{k},\varphi
V_{k}=U_{k}$,$\varphi \xi=0$ and the pseudo-Riemannian metric
$g=\eta_{\mathcal{H}^{2n+1}}\otimes
\eta_{\mathcal{H}^{2n+1}}+\sum_{k=1}^{n}((du_k)^2-(dv_k)^2)=d\eta_{\mathcal{H}^{2n+1}}(\cdot,\varphi\cdot)$.
Then the vector fields $\{U_{1},...U_{n},V_{1}...V_{n},\xi \}$
form an orthonormal basis for the paracontact metric structure,
$g(U_i,U_i)=-g(V_i,V_i)=1=g(\xi,\xi)$, $i=1,..,n$. It is easy to
check that $N_{\varphi}-2d\eta\otimes\xi=0$ and hence the
paracontact metric structure is a paraSasakian structure.

As another example \cite{S1}, take $\{x_{0},y_{0},\dots
,x_{n},y_{n}\}$ to be the standard coordinate system in $\mathbb
R^{2n+2}$. The standard para-K\"{a}hler structure $(\mathbb I,g)$
is defined by
$$\mathbb I\frac{\partial}{\partial x_i}=\frac{\partial}{\partial y_i},\quad \mathbb I\frac{\partial}{\partial y_i}=\frac{\partial}{\partial x_i},\quad g(\frac{\partial}{\partial x_i},\frac{\partial}{\partial x_j})=-g(\frac{\partial}{\partial y_i},\frac{\partial}{\partial y_j})=\delta_{ij},\quad g(\frac{\partial}{\partial x_i},\frac{\partial}{\partial y_j})=0,$$
where $i,j=0,...,n$. The Levi-Civita connection preserves the
fundamental 2-form defined by $g(\cdot,\mathbb I \cdot)$. Consider
the hypersurface
\begin{equation*}
\mathbf{H}^{2n+1}_{n+1}(1)= \{ (x_{0},y_{0},\dots ,x_{n},y_{n})\in
\mathbb R^{2n+2}\, |\ x_{0}^{2}+\dots +x_{n}^{2}-y_{0}^{2}\dots
-y_{n}^{2}=1 \}.
\end{equation*}

$\mathbf{H}^{2n+1}_{n+1}(1)$ carries a natural paracontact metric
structure inherited from the standard embedding in the
para-K\"{a}hler manifold $(\mathbb R^{2n+2},$ $\mathbb I,g)$. We
take
\begin{equation*}
\eta_{\mathbf{H}^{2n+1}_{n+1}(1)} =\sum_{j=0}^{n}\left(
x_{j}dy_{j}-y_{j}dx_{j}\right)
\end{equation*}
noting that here $N=\sum_{j=0}^{n}\left( x_{j}\frac{\partial
}{\partial x_{j} }+y_{j}\frac{\partial }{\partial y_{j}}\right)$
is the unit normal vector field and $\xi =\mathbb I
N=\sum_{j=0}^{n}\left( x_{j}\frac{\partial }{\partial
y_{j}}+y_{j}\frac{
\partial }{\partial x_{j}}\right) .$
For any vector field $X$ tangent to $\mathbf{H}^{2n+1}_{n+1}(1)$
we put $\varphi X=\mathbb I
X-\eta_{\mathbf{H}^{2n+1}_{n+1}(1)}(X)N$. Then we have
$\varphi^2X=X-\eta_{\mathbf{H}^{2n+1}_{n+1}(1)}(X)\xi$, $\varphi
\xi=0$, $\eta_{\mathbf{H}^{2n+1}_{n+1}(1)}(\varphi X)=0$ and
$\eta_{\mathbf{H}^{2n+1}_{n+1}(1)}(X)=g(\xi,X)$. The hyperboloid
$\mathbf{H}^{2n+1}_{n+1}(1)$ is a totally umbilical hypersurface
with the second fundamental form equal to minus identity, so we
obtain a paraSasakian structure,
$d\eta_{\mathbf{H}^{2n+1}_{n+1}(1)}(X,Y)=g(X,\varphi Y)$.

The above construction does not work for hyperboloids
$\mathbf{H}^{2n+1}_{n+1}(r)$ for $r$ different from 1. In fact,
for such hyperboloids we shall obtain a paracontact metric
structure, but $d\eta_{\mathbf{H}^{2n+1}_{n+1}(r)}$ would be equal
to $\frac{1}{r}g(\cdot,\varphi\cdot)$. However this flaw could be
rectified by a $\mathcal{D}$-homothetic transformation (see
\ref{d1} below).

The main purpose of this article is to prove the following
\begin{thm}\label{t20}
Let $M(\varphi,\xi,\eta,g)$ be a $(2n+1)$-dimensional $(n>1)$
connected, simply connected paraSasakian manifold with constant
paraholomorphic sectional curvature $k$.
\begin{enumerate}
\item[i).] If $k=3$, then $M$ is locally isomorphic to the
hyperbolic Heisenberg group $\mathcal{H}^{2n+1}$;

\item[ii).] If $k\not=3$, then $M$ is locally isomorphic to the
hyperboloid $\mathbf{H}(k)$, obtained from
$\mathbf{H}^{2n+1}_{n+1}(1)$ through $\mathcal{D}$-homothetic
transformation, for which
$\eta_{\mathbf{H}(k)}=-\dfrac{4}{k-3}\eta_{\mathbf{H}^{2n+1}_{n+1}(1)}$.
\end{enumerate}
\end{thm}

\textbf{Acknowledgement} I would like to thank Stefan Ivanov for
his useful comments. This work was partially supported by Contract
082/2009 with the University of Sofia "St. Kl. Ohridski". The
author also acknowledges support from the European Operational
program HRD through contract BGO051PO001/07/3.3-02/53 with the
Bulgarian Ministry of Education.

\section{Paracontact metric manifolds}
A smooth pseudo-Riemannian manifold $M^{(2n+1)}$ is said to be a
\emph{paracontact metric manifold} if it is equipped with a tensor
field $\varphi$ of type $(1,1)$, a vector field $\xi$, a 1-form
$\eta$ and pseudo-Riemannian metric $g$ satisfying the following
compatibility conditions
\begin{eqnarray}
  \label{f82}
    & &
    \begin{array}{cl}
          (i)   & \varphi(\xi)=0,\quad \eta \circ \varphi=0,\quad
          \\[5pt]
          (ii)  & \eta (\xi)=1 \quad \varphi^2 = id - \eta \otimes
          \xi.
             \\[5pt]
          (iii) & g(\varphi X,\varphi Y)=-g(X,Y)+\eta (X)\eta (Y)
          \\[3pt]
           (iv) & g(X,\varphi Y)=d\eta(X,Y),
    \end{array}
\end{eqnarray}
where $d\eta(X,Y)=\frac12(X\eta(Y)-Y\eta(X)-\eta([X,Y])$.

Let $\mathbb D=Ker~\eta$ be the  horizontal distribution generated
by $\eta$. The tensor field $\varphi$ induces an almost
paracomplex structure on each fibre on $\mathbb D$.

Recall that an almost paracomplex structure on a 2n-dimensional
manifold is a (1,1)-tensor $J$ such that $J^2=1$ and the
eigensubbundles $T^+,T^-$, corresponding to the eigenvalues $1,-1$
of $J$ respectively, have equal dimension $n$.

The Nijenhuis tensor $N_J$ of $J$, given by
$N_{J}(X,Y)=[JX,JY]-J[JX,Y]-J[X,JY]+[X,Y],$ is the obstruction for
the integrability of the eigensubbundles $T^+,T^-$. If $N_J=0$
then the almost paracomplex structure is called paracomplex or
integrable.

We recall that the paracontact metric manifold
$M^{(2n+1)}(\varphi,\xi,\eta,g)$ is said to be a
$\emph{paraSasakian manifold}$, if
$N_{\varphi}(X,Y)=2d\eta(X,Y)\xi$.

\begin{rem}\cite{Z1}\label{r1}
A paracontact metric manifold $M^{2n+1}(\varphi,\xi,\eta,g)$ is
paraSasakian if and only if
\begin{equation}\label{f6}
(\nabla_X\varphi)Y=-g(X,Y)\xi+\eta (Y)X,
\end{equation}
where $\nabla$ denotes the covariant differentiation with respect
to the Levi-Civita connection determined by $g$.
\end{rem}

Since $\mathbb D$ is orientable by the paracomplex structure
$\varphi$, the manifold $M$ is orientable exactly when the
canonical line bundle $E=\{\eta\in \Lambda^1: Ker~\eta=\mathbb
D\}$ is orientable. Any two contact forms $\bar\eta,\eta\in E$ are
connected by
\begin{equation}\label{gauge}
\overline{\eta}=\alpha\eta,
\end{equation}
for some non-vanishing positive smooth function $\alpha$ on $M$.

Let $M^{2n+1}(\varphi,\xi,\eta,g)$ be a paracontact metric
manifold. We consider a paracontact form $\overline{\eta}=\alpha
\eta$ and define the structure tensors
$(\overline{\varphi},\overline{\xi},\overline{g})$ corresponding
to $\overline{\eta}$ using the condition:

$(\star)$ For each point $x$ of $M$, the action of $\varphi$ and
$\overline{\varphi}$ are identical  on $\mathbb{D}_x$

Under condition $(\star)$ the transformation $\eta \rightarrow
\overline{\eta}=\alpha \eta$ induces a transformation of the
structure tensors by
\begin{eqnarray}
\label{f60}
   & &
\begin{array}{cl}
     (i) & \overline{\xi}=\frac{1}{\alpha}(\xi+\zeta),\quad \zeta
         =-\frac{1}{2\alpha}\varphi X(\alpha),\\
    \\[5pt]
    (ii) & g(X,\overline{\varphi}Y)=g(X,\varphi Y)+\frac{1}{2\alpha}(X(\alpha)-\xi(\alpha) \cdot \eta(X))\eta(Y),\\
     \\[5pt]
   (iii) & \overline{g}(X,Y)=\alpha(g(X,Y)-\eta(X)Y(\zeta)-\eta(Y)X(\zeta))+
\alpha(\alpha-1+|\zeta|^2)\eta(X)\eta(Y).
\end{array}
\end{eqnarray}

We call the transformation of the structure tensors given by
$(\ref{f60})$ a \emph{gauge transformation of paracontact
pseudo-Riemannian structure} (or paracontact conformal
transformation)\cite{Z1}. When $\alpha$ is constant the
transformation is said to be $\mathcal{D}$-homothetic. More
explicitly a $\mathcal{D}$-homothetic transformation is defined by
\begin{equation}\label{d1}
\overline{\varphi}=\varphi,\quad
\overline{\xi}=\frac{1}{\alpha}\xi, \quad \overline{\eta}=\alpha
\eta,\quad \overline{g}=\alpha g +(\alpha^2-\alpha)\eta \otimes
\eta,\quad \alpha=const>0.
\end{equation}

$Theorem~4.8$ in \cite{Z1} shows that the $\mathcal{D}$-homothetic
transformations preserve the paraSasakian structure.

\begin{defn}\cite{Z1}
A paracontact structure $(\eta,\varphi,\xi)$ is said to be
\emph{integrable} if the almost paracomplex structure
$\varphi_{|_{\mathbb D}}$ satisfies the conditions
\begin{equation}\label{new1}
N_{\varphi}(X,Y)=0, \quad [\varphi X,Y]+[X,\varphi Y] \in
\Gamma(\mathbb D)\quad X,Y\in \Gamma(\mathbb D).
\end{equation}
\end{defn}
\begin{defn}\cite{S1}
A paracontact manifold with an integrable paracontact structure is
called a $\emph{para CR-manifold}$.
\end{defn}

In order to prove the main theorem we shall need the notions of
the canonical paracontact connection $\widetilde{\nabla}$ and of
the tensor field $h$ defined in \cite{Z1}. Let $M$ be a
paracontact metric manifold. We define
\begin{equation}
h=\frac{1}{2}\pounds_{\xi}\varphi
\end{equation}
\begin{equation}\label{tan-web}
\begin{aligned}
\widetilde{\nabla}_XY=\nabla_XY+\eta(X)\varphi Y+\eta(Y)(\varphi
X-\varphi hX)+
\end{aligned}
\end{equation}
$$+g(X,\varphi Y)\cdot \xi-g(hX,\varphi Y)\cdot \xi,$$
where $X,Y\in \Gamma(TM)$.

The tensor field $h$ has the properties:
\begin{equation}\label{f51}
\nabla_X\xi=- \varphi X+\varphi hX,\quad \varphi h+h\varphi,\quad
trh=h\xi =0,\quad X\in \Gamma(TM).
\end{equation}

The torsion of the connection $\widetilde{\nabla}$ is given by
\begin{gather}\label{tprtw}
T(X,Y)=\eta(X)\varphi hY-\eta(Y)\varphi hX+2g(X,\varphi
Y)\xi,\quad X,Y\in \Gamma(TM).
\end{gather}

On a paracontact metric manifold the connection
$\widetilde{\nabla}$ has the properties
\begin{equation}\label{tnweb}
\begin{aligned}
\widetilde{\nabla}\eta=0,\quad \widetilde{\nabla}\xi=0,\quad \widetilde{\nabla}g=0,\\
(\widetilde{\nabla}_X\varphi)Y=(\nabla_X\varphi)Y+g(X-hX,Y)\xi-\eta(Y)(X-hX),\\
T(\xi,\varphi Y)=-\varphi T(\xi,Y),\quad Y \in \Gamma(\mathbb{D})\quad or \quad Y\in \Gamma(TM)\\
T(X,Y)=2d\eta(X,Y)\xi=2g(X,\varphi Y),\quad X,Y\in
\Gamma(\mathbb{D}),
\end{aligned}
\end{equation}
where $T(\xi,Y)=h(Y)$ is the horizontal part of the torsion and
$T(X,Y)=2d\eta(X,Y)\xi$ is the vertical part of the torsion.

For example, the hyperbolic Heisenberg group is flat with respect
to the canonical paracontact connection.

The curvature tensor
$\widetilde{R}=[\widetilde{\nabla},\widetilde{\nabla}]-\widetilde{\nabla}_{[,]}$
of $\widetilde{\nabla}$ is related to the curvature
$R=[\nabla,\nabla]-\nabla_{[,]}$  of the Levi-Civita connection by
\begin{equation}\label{f21}
\widetilde{R}_{ijk}^l=R_{ijk}^l+\nabla_i\varphi_k^l\eta_j-\nabla_j\varphi_k^l\eta_i+2\varphi_{ij}\varphi_k^l-
\varphi_s^l\nabla_j\xi^s\eta_i\eta_k+\varphi_s^l\nabla_i\xi^s\eta_j\eta_k+
\end{equation}
$$+\xi^l\nabla_i\eta_s\varphi_k^s\eta_j-\xi^l\nabla_j\eta_s\varphi_k^s\eta_i-\xi^lR_{ijk}^s\eta_s-
\eta_kR_{ijs}^l\xi^s+\nabla_j\eta_k\nabla_i\xi^l-\nabla_i\eta_k\nabla_j\xi^l.$$

Consequently, for the Ricci tensors $\widetilde{\mathbf{r}}$ of
$\widetilde{\nabla}$ and $\mathbf{r}$ of $\nabla$ we obtain
\begin{equation}\label{f50}
\widetilde{\mathbf{r}}_{jk}=\mathbf{r}_{jk}
-2g_{jk}+2\eta_j\eta_k-\mathbf{r}_{js}\eta_k\xi^s-R_{jsrk}\xi^s\xi^r-
\nabla_r\eta_k\nabla_j\xi^r.
\end{equation}

The canonical paracontact connection $\widetilde{\nabla}$
preserves the structural tensors $g$, $\xi$ and $\eta$. However it
does not preserve the tensor $\varphi$. More precisely we have
(Theorem~4.10 \cite{Z1}):
$$\widetilde{\nabla}\varphi=0\iff M \hbox{ is para CR-manifold.}$$

Let us require $M$ to be a paraSasakian manifold. According to
\cite{Z1}, we have $h=0$. Now, from $Remark~\ref{r1}$, $h=0$ and
the equalities $(\ref{tnweb})$, we obtain
\begin{equation}\label{tnweb1}
\begin{aligned}
\widetilde{\nabla}\eta=0,\quad \widetilde{\nabla}\xi=0,\quad
\widetilde{\nabla}g=0,\quad \widetilde{\nabla}\varphi=0.
\end{aligned}
\end{equation}
Hence, $M$ is a para CR-manifold.

$Theorem~4.11$ in \cite{Z1} shows that the horizontal part
($T(\xi,Y)=h(Y)$ (for $Y \in \Gamma(\mathbb{D})$) of the torsion
vanishes if and only if $M$ is a paraSasakian manifold.

\section{Identities}

Let us fix local coordinates $(x^1,\dots,x^{2n+1})$. We shall use
the Einstein summation convention. In local coordinates the
equations $(\ref{f82})$ take the form
$$\eta_r\xi^r=1,\qquad \varphi_r^i\xi^r=0,\qquad \eta_r\varphi^r_j=0,\qquad
\varphi^i_r\varphi^r_j=\delta^i_j-\xi^i\eta_j,$$
$$g_{rs}\varphi^r_j\varphi^s_k=-g_{jk}+\eta_j\eta_k,\qquad g_{jr}\xi^r=\eta_j.$$

Let $M^{(2n+1)}(\varphi,\xi,\eta,g)$ be a paraSasakian manifold.
Then we have the following identities (\cite{Z1})
\begin{equation}\label{f1}
\nabla_i\eta_j=\varphi_{ij}\quad \Longleftrightarrow \quad
\nabla_i\xi^j=-\varphi^j_i
\end{equation}
\begin{equation}\label{f2}
\nabla_r\varphi_{si}=\eta_ig_{rs}-\eta_sg_{ri}\quad
\Longleftrightarrow \quad
\nabla_r\varphi^s_i=\eta_i\delta^s_r-\xi^sg_{ri},
\end{equation}
where $\varphi_{ij}=-\varphi_{ji}=g_{il}\varphi^l_j$. In this
case, the vector field $\xi$ is a unit Killing vector field.

From $(\ref{f1})$  and $(\ref{f2})$, we have
\begin{equation}\label{f3}
\nabla_k\nabla_i\eta_j=g_{ik}\xi^j-\eta_ig_{kj}\quad or \quad
\nabla_k\nabla_i\xi^j=g_{ik}\xi^j-\eta_i\delta^j_k.
\end{equation}
Using the Ricci formula, we get
\begin{equation}\label{f4}
R_{kisl}\xi^l=g_{ks}\eta_i-g_{is}\eta_k\quad or \quad
R^l_{kis}\eta_l=g_{ks}\eta_i-g_{is}\eta_k.
\end{equation}
From $(\ref{f2})$ and the Ricci identities, we obtain
\begin{equation}\label{f5}
\varphi^a_j\varphi^b_iR_{ablk}=-R_{jilk}-\varphi_{ik}\varphi_{lj}+\varphi_{il}\varphi_{kj}-g_{kj}g_{il}+g_{lj}g_{ik}
\end{equation}
Transvecting $(\ref{f5})$ with $\varphi^j_m\varphi^l_h$, we have
\begin{equation}\label{f6}
\varphi^b_m\varphi^l_hR_{bilk}-\varphi^b_i\varphi^l_hR_{bmlk}=g_{km}g_{ih}-g_{mh}g_{ik}+g_{ik}\eta_m\eta_h-
\end{equation}
$$-g_{mk}\eta_i\eta_h+\varphi_{hm}\varphi_{ik}-\varphi_{km}\varphi_{ih}.$$
From
$$\varphi^b_i\varphi^l_hR_{bmlk}=-\varphi^b_i\varphi^l_h(R_{mlbk}+R_{lbmk})=\varphi^b_i\varphi^l_hR_{blmk}-\varphi^b_i\varphi^l_hR_{mlbk}$$
and $(\ref{f5})$, we get
\begin{equation}\label{f7}
\varphi^b_i\varphi^l_hR_{bmlk}-\varphi^b_h\varphi^l_iR_{bmlk}=-R_{ihmk}-g_{ki}g_{mh}+g_{mi}g_{hk}+\varphi_{hm}\varphi_{ki}-\varphi_{hk}\varphi_{mi}.
\end{equation}

\section{Paraholomorphic sectional curvature}\label{gau}

Let $M^{2n+1}(\varphi,\xi,\eta,g)$ be a paraSasakian manifold. Let
$p$ be any point on $M$. We consider $\mathbb D$ as a subspace of
$T_p(M)$, the tangent space of $M$ at $p$, whose elements are
orthogonal to $\xi$, i.e.,
$$\mathbb {D}=Ker~\eta=\{\upsilon \in T_p(M)| g(\xi,\upsilon)=0\}.$$

Let $u$ be any unit vector of $\mathbb D$, i.e.,
$g(u,u)=\varepsilon_{u}$, where $\varepsilon_{u}=\pm 1$. By
$\xi-$section, we mean the section determined by $\xi$ and $u$.

The sectional curvature determined by the $\xi-$section is said to
be \emph{$\xi$-sectional curvature}. Denoting it by $k(\xi,u)$, we
get
$$k(\xi,u)=\varepsilon_{u}\varepsilon_{\xi}R_{kjhi}u^k\xi^j\xi^hu^i=\frac{R_{kjhi}u^k\xi^j\xi^hu^i}{g_{ki}u^ku^ig_{jh}\xi^j\xi^h}.$$
Using $(\ref{f3})$ and $\varepsilon_{\xi}=1$, we get
$$k(\xi,u)=-\varepsilon_{u}(g_{ki}\eta_j-g_{ij}\eta_k)u^k\xi^ju^i.$$
By assumption,
$$g_{ij}\xi^iu^j=0\quad and \quad g_{ik}u^iu^k=\varepsilon_{u},$$
and hence we have
$$k(\xi,u)=-1.$$
Thus we obtain the following
\begin{pro}
In a paraSasakian manifold the $\xi-$sectional curvature is always
equal to -1.
\end{pro}
Next we consider the sectional curvature determined by two
orthogonal vectors in $\mathbb D$.

If $\upsilon$ is any unit vector (not necessarily in $\mathbb D$)
such that $\varphi \upsilon$ is not isotropic, then $\varphi
\upsilon$ lies on $\mathbb D$. Furthermore any element of $\mathbb
D$ can be written in this form.

If the section is determined by $\varphi \upsilon$ and
$\varphi^2\upsilon$, which are orthogonal to each other, we call
the section \emph{paraholomorphic}, and the sectional curvature
determined by it --- \emph{paraholomorphic sectional curvature}.
In such case we have
\begin{equation}\label{f8}
k=\frac{R_{dcab}(\varphi
\upsilon)^d(\varphi^2\upsilon)^c(\varphi^2\upsilon)^a(\varphi
\upsilon)^b}{g_{dc}(\varphi \upsilon)^d(\varphi
\upsilon)^cg_{ba}(\varphi^2\upsilon)^b(\varphi^2\upsilon)^a}
\end{equation}
\begin{thm}\label{t11}
Let $M^{2n+1}(\varphi,\xi,\eta,g)$ be a paraSasakian manifold of
dimension $2n+1$, $n>1$. If the paraholomorphic sectional
curvature is independent on the paraholomorphic section at a
point, then the curvature tensor has the form
\begin{equation}\label{f20}
R_{mjhl}=\frac{k-3}{4}(g_{ml}g_{hj}-g_{jl}g_{hm})+\frac{k+1}{4}(g_{hm}\eta_j\eta_l+g_{lj}\eta_m\eta_h-g_{lm}\eta_j\eta_h-
\end{equation}
$$-g_{hj}\eta_l\eta_m+\varphi_{hj}\varphi_{ml}-\varphi_{hm}\varphi_{jl}+2\varphi_{mj}\varphi_{hl}),$$
where $k$ is a constant.
\end{thm}
\begin{proof}
From identity $(\ref{f8})$, using $(\ref{f3})$, we have
\begin{equation}\label{f9}
k=-\frac{(\varphi^d_k\varphi^b_iR_{djhb}-\eta_j\eta_h(g_{ki}-\eta_k\eta_i))\upsilon^k\upsilon^j\upsilon^h\upsilon^i}{(g_{kj}-\eta_k\eta_j)(g_{ih}-\eta_i\eta_h)\upsilon^k\upsilon^j\upsilon^h\upsilon^i}.
\end{equation}
Now if we assume that $k$ is independent on the choice of the
paraholomorphic section at $p\in M$, then $(\ref{f9})$ is
equivalent to
\begin{equation}\label{f10}
(\varphi^d_k\varphi^b_iR_{djhb}-\eta_j\eta_h(g_{ki}-\eta_k\eta_i)+k(g_{kj}-\eta_k\eta_j)(g_{ih}-\eta_i\eta_h))\upsilon^k\upsilon^j\upsilon^h\upsilon^i=0
\end{equation}
for any vector $\upsilon$. From here, we obtain
$$2(\varphi^d_k\varphi^b_iR_{djhb}+\varphi^d_j\varphi^b_iR_{dhkb}+\varphi^d_h\varphi^b_iR_{dkjb})=-2(\varphi_{ik}\varphi_{jh}+\varphi_{ij}\varphi_{hk}+\varphi_{ih}\varphi_{kj})-2k(g_{kj}g_{ih}+$$
$$+g_{hj}g_{ki}+g_{ij}g_{kh})-6(k+1)\eta_i\eta_j\eta_k\eta_h+(1+2k)(g_{ij}\eta_k\eta_h+g_{ih}\eta_k\eta_j+g_{ik}\eta_j\eta_h+g_{kh}\eta_i\eta_j+$$
$$+g_{hj}\eta_k\eta_i+g_{kj}\eta_i\eta_h)+(g_{ij}\eta_k\eta_h+g_{ih}\eta_k\eta_j+g_{ik}\eta_j\eta_h-g_{kh}\eta_i\eta_j-g_{hj}\eta_k\eta_i-g_{kj}\eta_i\eta_h)$$
by virtue of $(\ref{f6})$.

Using $(\ref{f9})$ and $(\ref{f10})$, we have
$$3\varphi^d_k\varphi^b_iR_{djhb}=R_{kijh}-R_{jihk}+3(g_{ki}g_{hj}-g_{ih}g_{kj})-k(g_{ji}g_{hk}+g_{kj}g_{hi}+g_{ki}g_{hj})+$$
$$+(k+1)(g_{ij}\eta_k\eta_h+g_{ih}\eta_k\eta_j+g_{ik}\eta_j\eta_h+g_{kh}\eta_i\eta_j+g_{kj}\eta_i\eta_h)+(k-2)g_{hj}\eta_k\eta_i-$$
$$-3\varphi_{ij}\varphi_{hk}-3(k+1)\eta_i\eta_j\eta_k\eta_h.$$
Contracting the last equation with $\varphi^i_m\varphi^k_l$ and
using $(\ref{f3})$ and $(\ref{f5})$, we obtain
$$\varphi^i_m\varphi^k_lR_{ijhk}=3R_{ljhm}+R_{lmjh}+(k+2)\varphi_{hm}\varphi_{jl}-(k+1)\varphi_{mj}\varphi_{hl}-(k-3)g_{ml}g_{hj}-$$
$$-2g_{mj}g_{hl}-g_{jl}g_{hh}-(k+1)\eta_m\eta_j\eta_l\eta_h+kg_{ih}\eta_l\eta_m+(k+1)g_{lm}\eta_h\eta_j.$$
The skew-symmetric part of the last equation with respect to $m$
and $j$, and $(\ref{f6})$ yield
\begin{equation}\label{f11}
4R_{mjhl}=(k-3)(g_{ml}g_{hj}-g_{jl}g_{hm})+(k+1)(g_{hm}\eta_j\eta_l+g_{lj}\eta_m\eta_h-g_{lm}\eta_j\eta_h-
\end{equation}
$$-g_{hj}\eta_l\eta_m+\varphi_{hj}\varphi_{ml}-\varphi_{hm}\varphi_{jl}+2\varphi_{mj}\varphi_{hl}).$$
Contracting the last equation with $g^{ml}$, we obtain
\begin{equation}\label{f12}
2\mathbf{r}_{jh}=[n(k-3)+k+1]g_{hj}-(n+1)(k+1)\eta_j\eta_h,
\end{equation}
where $\mathbf{r}_{ij}=g^{kl}R_{kijl}$ is the Ricci tensor.
\footnote{From here, $M$ is an $\eta-$Einstein manifold (see
$Definition~\ref{eta}$).}

Moreover contracting the last equation with $g^{jh}$, we have
\begin{equation}\label{f13}
2\mathbf{s}=n(2n+1)(k-3)+n(k+1),
\end{equation}
where $\mathbf{s}=g^{ij}\mathbf{r}_{ij}$ is the scalar curvature.
On the other hand, from the second Bianchi identity, we get
\begin{equation}\label{f14}
\nabla_i(\mathbf{s})-2g^{jh}\nabla_j(\mathbf{r}_{ih})=0.
\end{equation}
Substituting $(\ref{f12})$ and $(\ref{f13})$ into $(\ref{f14})$,
we obtain
\begin{equation}\label{f15}
(n-1)\nabla_i(\mathbf{s})+\eta_i\xi^j\nabla_j(\mathbf{s})=0.
\end{equation}
Transvecting $(\ref{f15})$ with $\xi^i$, we get
$$\xi^j\nabla_j(\mathbf{s})=0$$
and hence $(\ref{f15})$ implies
$$\nabla_i(\mathbf{s})=0\quad (n>1).$$
Thus, the paraholomorphic sectional curvature $k$ is a constant.
\end{proof}
A paraSasakian manifold is said to be of \emph{constant
paraholomorphic curvature}, denoted by $M(k)$, if the
paraholomorphic sectional curvature $k$ is constant on the
manifold.

\begin{thm}\label{t13}
Let $M^{2n+1}(\varphi,\xi,\eta,g)$ be a paraSasakian manifold with
constant paraholomorphic sectional curvature $k$. Then applying a
$\mathcal{D}$-homothetic transformation on $M$, we obtain a
paraSasakian structure
$(\overline{\varphi},\overline{\xi},\overline{\eta},\overline{g})$
on $M$, for which the paraholomorphic sectional curvature is
constant and equal to $\frac{k-3}{\alpha}+3$.
\end{thm}
\begin{proof}
Let $M^{2n+1}(\varphi,\eta,\xi,g)$ be a paraSasakian manifold with
constant paraholomorphic sectional curvature $k$. Suppose that
$M^{2n+1}(\overline{\varphi},\overline{\eta},\overline{\xi},\overline{g})$
is obtained from $M^{2n+1}(\varphi,\eta,\xi,g)$ through a
$\mathcal{D}$-homothetic transformation. Using the formulas
\cite{Z1}, we get
\begin{equation}\label{f32}
\overline{g}_{jk}=\alpha g_{jk}+\alpha(\alpha-1) \eta_j \eta_k,
\end{equation}
\begin{gather}\label{f33}
\overline{R}_{ijk}^l=R_{ijk}^l-(\alpha-1)(2\varphi_k^l\varphi_{ij}-
\varphi_i^l\varphi_{jk}+\varphi_j^l\varphi_{ik})+
\\\nonumber+(\alpha-1)(\nabla_j\varphi_i^l\eta_k+\nabla_j\varphi_k^l\eta_i-
\nabla_i\varphi_j^l\eta_k-\nabla_i\varphi_k^l\eta_j)+
(\alpha-1)^2(\delta_j^l\eta_i\eta_k-\delta_i^l\eta_j\eta_k),
\end{gather}
\begin{equation}\label{f34}
\mathbf{\overline{r}}_{jk}=\mathbf{r}_{jk}+2(\alpha-1)g_{jk}-2(\alpha-1)((2n+1)+n(\alpha-1))\eta_j\eta_k,
\end{equation}
\begin{equation}\label{f35}
\mathbf{\overline{s}}=\frac{1}{\alpha}\mathbf{s}+2n\frac{\alpha-1}{\alpha}.
\end{equation}
Since $M^{2n+1}(\varphi,\eta,\xi,g)$ is paraSasakian, the
identities $(\ref{f1})$, $(\ref{f2})$ imply
\begin{gather}\label{f100}
\overline{R}_{ijk}^l=R_{ijk}^l-(\alpha-1)(2\varphi_k^l\varphi_{ij}-
\varphi_i^l\varphi_{jk}+\varphi_j^l\varphi_{ik})+
(\alpha^2-1)(\delta_j^l\eta_i\eta_k-\delta_i^l\eta_j\eta_k),
\end{gather}
From $(\ref{f20})$, $(\ref{f32})$ and $(\ref{f100})$ we get
\begin{equation}\label{f101}
\overline{R}_{mjhl}=\frac{\overline{k}-3}{4}(\overline{g}_{ml}\overline{g}_{hj}-\overline{g}_{jl}\overline{g}_{hm})+\frac{\overline{k}+1}{4}(\overline{g}_{hm}\overline{\eta}_j\overline{\eta}_l+\overline{g}_{lj}\overline{\eta}_m\overline{\eta}_h-\overline{g}_{lm}\overline{\eta}_j\overline{\eta}_h-
\end{equation}
$$-\overline{g}_{hj}\overline{\eta}_l\overline{\eta}_m+\overline{\varphi}_{hj}\overline{\varphi}_{ml}-\overline{\varphi}_{hm}\overline{\varphi}_{jl}+2\overline{\varphi}_{mj}\overline{\varphi}_{hl}),$$
where $\overline{k}=\frac{k-3}{\alpha}+3$.
\end{proof}
The hyperboloid $\mathbf{H}^{2n+1}_{n+1}(1)$ has sectional
curvature equal to minus one \cite{W1}. In particular, its
paraholomorphic sectional curvature is also minus one. Thus,
applying a $\mathcal{D}$-homothetic transformation on
$\mathbf{H}^{2n+1}_{n+1}(1)$ (i.e., $\overline{\eta}=\alpha\eta$),
we obtain a paraSasakian manifold with constant paraholomorphic
sectional curvature $k=-\frac{4}{\alpha}+3$. Let us denote this
manifold by $\mathbf{H}(k)$.

\section{The PC-Bochner curvature tensor}

For a paraSasakian manifold we define the tensor $\mathbf{B}$
\begin{equation}\label{f30}
\mathbf{B}_{ijkl}=R_{ijkl}+\frac{1}{2n+4}(\mathbf{r}_{ik}g_{jl}-\mathbf{r}_{jk}g_{il}+\mathbf{r}_{jl}g_{ik}-\mathbf{r}_{il}g_{jk}+\mathbf{r}_{sk}\varphi_i^s\varphi_{jl}-
\end{equation}
$$-\mathbf{r}_{sk}\varphi_j^s\varphi_{il}+\mathbf{r}_{sl}\varphi_j^s\varphi_{ik}-\mathbf{r}_{sl}\varphi_i^s\varphi_{jk}+2\mathbf{r}_{sj}\varphi_i^s\varphi_{kl}+2\mathbf{r}_{sl}\varphi_k^s\varphi_{ij}-\mathbf{r}_{ik}\eta_j\eta_l+$$
$$+\mathbf{r}_{jk}\eta_i\eta_l-\mathbf{r}_{jl}\eta_i\eta_k+\mathbf{r}_{il}\eta_j\eta_k)-\frac{k+2n}{2n+4}(\varphi_{ik}\varphi_{jl}-\varphi_{jk}\varphi_{il}+2\varphi_{ij}\varphi_{kl})+$$
$$+\frac{k-4}{2n+4}(g_{ik}g_{jl}-g_{jk}g_{il})-\frac{k}{2n+4}(g_{ik}\eta_j\eta_l-g_{jk}\eta_i\eta_l+g_{jl}\eta_i\eta_k-g_{il}\eta_j\eta_k),$$
where $k=-\frac{\mathbf{s}-2n}{2n+2}$.

By straightforward computations from the identity $(\ref{f30})$ we
obtain the following
\begin{lem} The tensor $\mathbf{B}$ has identities
\begin{equation}\label{31}
\mathbf{B}_{ijkl}=-\mathbf{B}_{jikl},\quad
\mathbf{B}_{ijkl}=\mathbf{B}_{klij},\quad
\mathbf{B}_{ijkl}+\mathbf{B}_{jkil}+\mathbf{B}_{kijl}=0,\quad
g^{il}\mathbf{B}_{ijkl}=0,
\end{equation}
$$\xi^i\mathbf{B}_{ijkl}=0,\quad \mathbf{B}_{sjkl}\varphi^s_i=-\mathbf{B}_{iskl}\varphi^s_j.$$
\end{lem}

\begin{lem}\label{t14}
The tensor $\mathbf{B}$ is invariant under
$\mathcal{D}$-homothetic transformations, i.e., if
$\overline{\eta}=\alpha\eta$, then
$\alpha^{-1}\overline{\mathbf{B}}=\mathbf{B}$.
\end{lem}
\begin{proof}
The claim of the Lemma is direct consequence of identities
$(\ref{f32}-\ref{f35})$.

\end{proof}
\begin{defn}
The tensor $\mathbf{B}$ is said to be paracontact Bochner
curvature tensor (PC-Bochner curvature tensor for short).
\end{defn}
\begin{defn}\label{eta}
If the Ricci tensor of a paraSasakian manifold $M$ is of the form
$$Ric(X,Y)=ag(X,Y)+b\eta(X)\eta(Y),$$
$a$ and $b$ being constant, then $M$ is called an
\emph{$\eta-$Einstein manifold}.
\end{defn}
Using the PC-Bochner curvature tensor we have
\begin{thm}\label{t15}
Let $M^{2n+1}(\varphi,\eta,\xi,g)$ be a paraSasakian manifold.
Then $M$ is with constant paraholomorphic sectional curvature if
and only if it is $\eta$-Einstein and with vanishing PC-Bochner
tensor.
\end{thm}
\begin{proof}
If $M$ is with a constant paraholomorphic sectional curvature $k$,
then from equality $(\ref{f20})$ we obtain
\begin{equation}\label{f40}
\mathbf{r}_{jh}=\frac{(k-3)n+k+1}{2}g_{jh}-\frac{(k+1)(n+1)}{2}\eta_j\eta_h,
\end{equation}
\begin{equation}\label{f41}
\frac{(k-3)n}{2}=\frac{\mathbf{s}-2n}{2n+2}
\end{equation}
and hence $M$ is $\eta$-Einstein.

Using $(\ref{f20})$ and $(\ref{f40})$ we have that
\begin{equation}\label{f42}
R_{ijkl}+\frac{1}{2n+4}(\mathbf{r}_{ik}g_{jl}-\mathbf{r}_{jk}g_{il}+\mathbf{r}_{jl}g_{ik}-\mathbf{r}_{il}g_{jk}+\mathbf{r}_{sk}\varphi_i^s\varphi_{jl}-
\end{equation}
$$-\mathbf{r}_{sk}\varphi_j^s\varphi_{il}+\mathbf{r}_{sl}\varphi_j^s\varphi_{ik}-\mathbf{r}_{sl}\varphi_i^s\varphi_{jk}+2\mathbf{r}_{sj}\varphi_i^s\varphi_{kl}+2\mathbf{r}_{sl}\varphi_k^s\varphi_{ij}-\mathbf{r}_{ik}\eta_j\eta_l+$$
$$+\mathbf{r}_{jk}\eta_i\eta_l-\mathbf{r}_{jl}\eta_i\eta_k+\mathbf{r}_{il}\eta_j\eta_k)=\frac{\frac{(3-k)n}{2}+2n}{2n+4}(\varphi_{ik}\varphi_{jl}-\varphi_{jk}\varphi_{il}+2\varphi_{ij}\varphi_{kl})-$$
$$-\frac{\frac{(3-k)n}{2}-4}{2n+4}(g_{ik}g_{jl}-g_{jk}g_{il})+\frac{\frac{(3-k)n}{2}}{2n+4}(g_{ik}\eta_j\eta_l-g_{jk}\eta_i\eta_l+g_{jl}\eta_i\eta_k-g_{il}\eta_j\eta_k).$$
From $(\ref{f42})$ and $(\ref{f41})$, we obtain $\mathbf{B}=0$.

Since $M$ is an $\eta$-Einstein manifold, we have that
\begin{equation}
\mathbf{r}_{jk}=ag_{jk}+b\eta_j\eta_k,
\end{equation}
where $a=\frac{\mathbf{s}}{2n}+1=const.$ and
$b=-\frac{\mathbf{s}}{2n}-(2n+1)=const.$ \cite{Z1}. Put
$\mathbf{B}=0$ into $(\ref{f30})$ to get
\begin{equation}\label{f36}
R_{mjhl}=\frac{k-3}{4}(g_{ml}g_{hj}-g_{jl}g_{hm})+\frac{k+1}{4}(g_{hm}\eta_j\eta_l+g_{lj}\eta_m\eta_h-g_{lm}\eta_j\eta_h-
\end{equation}
$$-g_{hj}\eta_l\eta_m+\varphi_{hj}\varphi_{ml}-\varphi_{hm}\varphi_{jl}+2\varphi_{mj}\varphi_{hl}),$$
where $k=\frac{\mathbf{s}+3n^2+n}{n(n+1)}$ is a constant.
\end{proof}
If $M$ is para CR-manifold, we may introduce the paracontact
conformal curvature tensor $W^{pc}$ \cite{S1}, which is an
analogue of the Chern-Moser tensor. It is defined by \cite{S1}
\begin{equation}
W^{pc}(X,Y,Z,W)=\widetilde{R}(X,Y,Z,W)-\frac{\widetilde{\mathbf{s}}}{4(n+1)(n+2)}(g(X,Z)g(Y,W)-g(Y,Z)g(X,W))+\\\nonumber
\end{equation}
$$+\frac{\widetilde{\mathbf{s}}}{4(n+1)(n+2)}(F(X,Z)F(Y,W)-F(Y,Z)F(X,W)+2F(X,Y)F(Z,W))+$$
$$+\frac{1}{2(n+2)}(g(X,Z)\widetilde{\mathbf{r}}(Y,W)-g(Y,Z)\widetilde{\mathbf{r}}(X,W)+g(Y,W)\widetilde{\mathbf{r}}(X,Z)-g(X,W)\widetilde{\mathbf{r}}(Y,Z))+$$
$$+\frac{1}{2(n+2)}(F(X,Z)\widetilde{\mathbf{r}}(Y,\varphi W)-F(Y,Z)\widetilde{\mathbf{r}}(X,\varphi W)+F(Y,W)\widetilde{\mathbf{r}}(X,\varphi Z)-F(X,W)\widetilde{\mathbf{r}}(Y,\varphi Z))$$
$$+\frac{1}{2(n+2)}(2F(X,Y)\widetilde{\mathbf{r}}(Z,\varphi W)+2F(Z,W)\widetilde{\mathbf{r}}(X,\varphi
Y)),$$ where $X,Y,Z,W \in \mathbb{D}$.

Note that the paracontact conformal curvature tensor $W^{pc}$ is
invariant under paracontact conformal transformations
($Theorem~1.1$ in \cite{S1}). $Theorem~1.2$ in \cite{S1} and
$Corollary~1.3$ in \cite{S1} show that if $M$ is a para
CR-manifold, the paracontact conformal curvature tensor $W^{pc}=0$
if and only if $M$ is locally paracontact conformal to the
standard flat paracontact metric structure, with respect to the
canonical paracontact connection $\widetilde{\nabla}$, on the
hyperbolic Heisenberg group $\mathcal{H}^{2n+1}$ or $M$ is locally
paracontact conformal to the hyperboloid
$\mathbf{H}^{2n+1}_{n+1}(1)$.

\begin{pro}
Let $M^{2n+1}(\varphi,\xi,\eta,g)$ be a paraSasakian manifold. If
the vector fields $X,Y,Z,W$ lie in $\mathbb {D}$, then
$\mathbf{B}(X,Y,Z,W)=W^{pc}(X,Y,Z,W)$.
\end{pro}
\begin{proof} For the paraSasakian manifold $M^{2n+1}(\varphi,\xi,\eta,g)$ from
identities (\ref{f21}) and (\ref{f50}) we obtain
\begin{equation}\label{f54}
\widetilde{R}(X,Y,Z,W)=R(X,Y,Z,W)-F(X,Z)F(Y,W)+F(Y,Z)F(X,W)-
\end{equation}
$$-2F(X,Y)F(Z,W)-g(X,Z)\eta(Y)\eta(W)+g(Y,Z)\eta(X)\eta(W)+$$
$$+g(X,W)\eta(Y)\eta(Z)-g(Y,W)\eta(X)\eta(Z),\quad X,Y,Z,W \in \Gamma(TM).$$
Consequently
\begin{equation}\label{f53}
\widetilde{\mathbf{r}}(X,Y)=\mathbf{r}(X,Y)-2g(X,Y)+2(n+1)\eta(X)\eta(Y),\quad
X,Y\in \Gamma(TM).
\end{equation}

Besides, $\widetilde{\mathbf{s}}=\mathbf{s}-2n$.

For $X,Y,Z,W \in \mathbb {D}$ the PC-Bochner curvature tensor has
the form
\begin{equation}\label{f57}
\mathbf{B}(X,Y,Z,W)=R(X,Y,Z,W)+\frac{1}{2(n+2)}(g(X,Z)\mathbf{r}(Y,W)-
\end{equation}
$$-g(Y,Z)\mathbf{r}(X,W)+g(Y,W)\mathbf{r}(X,Z)-g(X,W)\mathbf{r}(Y,Z))+F(X,Z)\mathbf{r}(Y,\varphi W)-$$
$$-F(Y,Z)\mathbf{r}(X,\varphi W)+F(Y,W)\mathbf{r}(X,\varphi Z)-F(X,W)\mathbf{r}(Y,\varphi Z)+2F(X,Y)\mathbf{r}(Z,\varphi W)+$$
$$+2F(Z,W)\mathbf{r}(X,\varphi Y))-\frac{k+2n}{2n+4}(F(X,Z)F(Y,W)-F(Y,Z)F(X,W)+$$
$$2F(X,Y)F(Z,W))+\frac{k-4}{2n+4}(g(X,Z)g(Y,W)-g(Y,Z)g(X,W)).$$

From equalities $(\ref{f54})$ and $(\ref{f53})$, for $X,Y,Z,W \in
\mathbb {D}$, we get
\begin{equation}\label{f55}
\widetilde{R}(X,Y,Z,W)=R(X,Y,Z,W)-F(X,Z)F(Y,W)+
\end{equation}
$$+F(Y,Z)F(X,W)-2F(X,Y)F(Z,W)$$
and
\begin{equation}\label{f56}
\widetilde{\mathbf{r}}(X,Y)=\mathbf{r}(X,Y)-2g(X,Y),\quad \quad
\widetilde{\mathbf{s}}=\mathbf{s}-2n.
\end{equation}
From $(\ref{f57})$, $(\ref{f55})$ and $(\ref{f56})$, we obtain
$\mathbf{B}(X,Y,Z,W)=W^{pc}(X,Y,Z,W)$.
\end{proof}

\section{Proof of Theorem~1.1}\label{gau1}

Suppose that $M^{2n+1}(\varphi,\xi,\eta,g)$ is a paraSasakian
manifold with constant paraholomorphic sectional curvature $k$.
From equalities $(\ref{f1})$, $(\ref{f2})$ and identities
$(\ref{f20})$, $(\ref{f21})$, we obtain
\begin{equation}\label{f22}
\widetilde{R}_{mjhl}=\frac{k-3}{4}(g_{ml}g_{hj}-g_{jl}g_{hm}+g_{hm}\eta_j\eta_l+g_{lj}\eta_m\eta_h-g_{lm}\eta_j\eta_h-
\end{equation}
$$-g_{hj}\eta_l\eta_m+\varphi_{hj}\varphi_{ml}-\varphi_{hm}\varphi_{jl}+2\varphi_{mj}\varphi_{hl}).$$
\begin{enumerate}
\item[i).] If $k=3$, then $\widetilde{R}_{mjhl}=0$ and Theorem~4.2
of \cite{S1} shows that $M$ is locally isomorphic to the
hyperbolic Heisenberg group $\mathcal{H}^{2n+1}$;
\end{enumerate}
\begin{enumerate}
\item[ii).] Let $k\not=3$. We shall adopt the following
convention: $x$, $\widetilde{\nabla}$, $\widetilde{R}$,\ldots
shall stand for points, connection and tensors of the manifold
$M$. On the other hand, $x^*$, $\widetilde{\nabla}^*$,
$\widetilde{R}^*$,\ldots shall stand for points, connection and
tensors of the manifold $\mathbf{H}(k)$.

First, let us note that $\widetilde{\nabla}T=0$ and
$\widetilde{\nabla}\widetilde{R}=0$. Indeed, the first one follows
from the fact that the horizontal part of the torsion $T$ is zero
and the identities $(\ref{tnweb1})$. The second one is yielded by
the identities $(\ref{tnweb1})$ and the fact that the curvature
tensor $\widetilde{R}$ has the form (\ref{f22}). Thus we have that
the manifold $M$ and the connection $\widetilde{\nabla}$ are
analytic \cite{K3}.

For any point $x$ of $M$ and $x^*$ of $\mathbf{H}(k)$, let
$(e_1,...,e_n,\varphi e_1,...,\varphi e_n,\xi)$ and
$(e^*_1,...,e^*_n,\varphi e^*_1,...,\varphi e^*_n,\xi^*)$ be
orthonormal basis at $x$ and $x^*$ respectively. We define a
linear isomorphism $F:T_x(M)\rightarrow T_{x^*}(\mathbf{H}(k))$ by
$Fe_i=e^*_i$, $F\varphi e_i=\varphi e^*_i$ ($i=1...n$) and
$F\xi=\xi^*$. Then we have $F\varphi=\varphi^*F$. Since both
paraholomorphic sectional curvatures are equal to $k$, $F$ maps
$\widetilde{R}$ into $\widetilde{R}^*$, F being consider as map of
tensor algebra. The covariant derivatives of $\varphi$ and $\xi$
are written in terms of $\varphi, \xi, g$ and hance the covariant
derivative of $R$ is expressed by $\varphi, \xi$ and $g$, that is,
$F$ maps $(\widetilde{\nabla}\widetilde{R})_x$ into
$(\widetilde{\nabla}^*\widetilde{R}^*)_{x^*}$. Likewise, we see
that $F$ maps $(\widetilde{\nabla^l}\widetilde{R})_x$ into
$(\widetilde{\nabla^l}^*\widetilde{R}^*)_{x^*}$ for every positive
integer $l$. Then we have an isomorphism $f$ of $M$ onto
$\mathbf{H}(k)$ such that $f(x)=x^*$ and the differential of $f$
at $x$ is $F$ (e.g. Theorem 7.2, p. 259 in \cite{K3}). We then
have that $(\widetilde{\nabla}\xi)_x$ is mapped to
$(\widetilde{\nabla}^*\xi^*)_{x^*}$. Thus we have
$$(\widetilde{\nabla}^*(f\xi))_{x^*}=f \cdot((\widetilde{\nabla}\xi)_x)=F((\widetilde{\nabla}\xi)_x)=(\widetilde{\nabla}^*\xi^*)_{x^*}.$$
Since $f$ is an isomorphism, $f\xi$ is also Killing vector field.
By $(f\xi)_{x^*}=\xi^*_{x^*}$ and
$\widetilde{\nabla}^*(f\xi))_{x^*}=(\widetilde{\nabla}^*\xi^*)_{x^*}$,
we get $f\xi=\xi^*$. Because $\varphi$ and $\eta$ ($\varphi^*$ and
$\eta^*$) are determined by $g$ and $\xi$ ($g^*$ and $\xi^*$), $f$
is an isomorphism between $M$ and $\mathbf{H}(k)$.

\end{enumerate}

\bibliographystyle{hamsplain}






\end{document}